\documentclass[12pt,leqno]{amsart}
\usepackage{a4,amssymb,amsthm,amscd,amsmath,xcolor,verbatim,url,enumerate,mathdots,cleveref}
\usepackage{color}

\renewcommand{\setminus}{\smallsetminus}

\DeclareMathOperator{\newsum}{\mathsf{\Sigma}}
\renewcommand{\sum}{\newsum}

\renewcommand{\inf}{\mathsf{inf}}
\renewcommand{\min}{\mathsf{min}}
\renewcommand{\max}{\mathsf{max}}

\renewcommand{\leq}{\leqslant}
\renewcommand{\geq}{\geqslant}
\renewcommand{\setminus}{\smallsetminus}

\swapnumbers
\numberwithin{equation}{section}
\newtheorem{thm}[equation]{Theorem}
\newtheorem*{thm*}{Theorem}
\newtheorem{prop}[equation]{Proposition}
\newtheorem{cor}[equation]{Corollary}
\newtheorem*{cor*}{Corollary}

\newtheorem*{conj*}{Conjecture}
\newtheorem{lem}[equation]{Lemma}

\newtheorem{qu}[equation]{Question}
\newtheorem*{qu*}{Question}

\theoremstyle{definition}
\newtheorem*{df}{Definition}
\newtheorem{ex}[equation]{Example}
\newtheorem{ex*}{Example}

\newtheorem{rem}[equation]{Remark}

\theoremstyle{plain}

\title[Sums of squares in function fields]{Sums of squares in function fields over henselian discretely valued fields}
\date{22.02.2023}

\author[G.~Manzano-Flores]{Gonzalo Manzano-Flores}

\address{Universidad de Santiago de Chile, Facultad de Ciencias, Avenida Libertador Bernardo O'Higgins nº 3363, Estaci\'on Central, Santiago, Chile.}
\email{gonzalo.manzano@usach.cl}

\begin{document}


\begin{abstract}

\medskip\noindent
Let $n\in\mathbb{N}$ and let $K$ be a field with a henselian discrete valuation of rank $n$ with hereditarily euclidean residue field. Let $F/K$ be a function field in one variable in one variable. It is known that every sum of squares is a sum of $3$ squares. We determine the order of the group of nonzero sums of $3$ squares modulo sums of $2$ squares in $F$ in terms of equivalence classes of certain discrete valuations of $F$ of rank at most $n.$ In the case of function fields of hyperelliptic curves of genus $g,$ K.J. Becher and J. Van Geel showed that the order of this quotient group is bounded by $2^{n(g+1)}.$ We show that this bound is optimal. Moreover, in the case where $n=1,$ we show that if $F/K$ is a hyperelliptic function field such that the order of this quotient group is $2^{g+1},$ then $F$ is nonreal. 


\medskip\noindent
{\sc Classification (MSC 2020):} 11E81, 12D15, 12J10
\end{abstract}

\maketitle

\section{Introduction}

Let $K$ be a field. For $d\in\mathbb{N},$ let $\mathsf{S}_{d}(K)$ denote the set of nonzero sums of $d$ squares in $K.$ Let $\mathsf{S}(K)=\displaystyle\bigcup_{n\in\mathbb{N}}\mathsf{S}_{n}(K),$ the multiplicative group of nonzero sums of squares in $K.$ The Pythagoras number $p(K)$ of a field $K$ is the smallest integer $d$ such that $\mathsf{S}(K)=\mathsf{S}_{d}(K)$ if such an integer $d$ exists, otherwise it is infinite. We call a finitely generated field extension of transcendence degree one a \textit{function field in one variable}, for short. In \cite{Witt34} it was shown that $p(F)=2$ for every function field in one variable $F$ over $\mathbb{R}.$ However, the same does not hold for function field in one variables over $\mathbb{R}(\!(t_{1})\!)\ldots(\!(t_{n})\!)$ when $n\geq 1.$ For example, S. Tikhonov in \cite[Example 3.8]{TVY06} showed that $tX$ is a sum of three squares but not a sum of two squares in the function field $F$ of the curve $Y^{2}=(tX-1)(X^{2}+1)$ over $\mathbb{R}(\!(t)\!),$ which shows that $p(F)>2.$ However, in that example it was not known whether $tX$ was the only nontrivial element of the quotient group of nonzero sums of squares in $F$ modulo the subgroup of nonzero sums of two squares in $F.$


Let $n\in\mathbb{N}.$ Consider a function field in one variable $F$ over the base field $\mathbb{R}(\!(t_{1})\!)\ldots(\!(t_{n})\!).$ It was shown in \cite[Theorem 6.12]{BGVG} that $p(F)\leq 3$ and that the quotient $\mathsf{S}(F)/\mathsf{S}_{2}(F),$ which measures the failure of every sum of squares to be a sum of two squares, is a finite group. In this article, we determine the order of $\mathsf{S}(F)/\mathsf{S}_{2}(F)$ in terms of the cardinality of a certain finite set of valuations on $F.$ In fact, we consider situations with slightly more general base fields (than $\mathbb{R}(\!(t_{1})\!)\ldots(\!(t_{n})\!)$) namely fields carrying a nontrivial henselian valuation with value group $\mathbb{Z}^{n}$ endowed with the lexicographic order, and such that the Pythagoras number of the rational function field over its residue field is bounded by $2^{r},$ for some $r\in\mathbb{N}^{+}.$ Concretely, for a function field in one variable $F/\mathbb{R}(\!(t_{1})\!)\ldots(\!(t_{n})\!),$ we show that $|\mathsf{S}(F)/\mathsf{S}_{2}(F)|= 2^{m},$ where $m$ is the number of equivalence classes of discrete valuations of rank at most $n$ and with nonreal residue field not containing $\sqrt{-1}.$ In this case, this generalizes \cite[Theorem 6.12]{BGVG}. It is natural to wonder whether we can bound the order of $\mathsf{S}(F)/\mathsf{S}_{2}(F)$ in terms of the genus of $F/K.$ In \cite{BG20}, reduction theory of arithmetic curves over $\mathbb{R}[\![t]\!]$ is applied to obtain the bound $|\mathsf{S}(F)/\mathsf{S}_{2}(F)|\leq 2^{g+1}$ for any function field in one variable $F$ of genus $g$ over $\mathbb{R}(\!(t)\!).$ In the case of a function field $F$ of a hyperelliptic curve of genus $g$ over $\mathbb{R}(\!(t_{1})\!)\ldots(\!(t_{n})\!),$ it is a consequence of \cite[Theorem 3.10]{BV09} that $$|\mathsf{S}(F)/\mathsf{S}_{2}(F)|\leq 2^{n(g+1)}.$$ It is known that in the case where $n=1$ and $g=1,$ the bound is optimal. In fact, for the nonreal function field $F$ of the curve $Y^{2}=-(X^{2}+1)(X^{2}+t^{2})$ over $\mathbb{R}(\!(t)\!),$ it was shown in \cite[Example 5.12]{BV09} that $|\mathsf{S}(F)/\mathsf{S}_{2}(F)|=2^{2}.$ In this article, for any $g\in\mathbb{N},$ we will construct a hyperelliptic curve of genus $g$ over $K$ such that $|\mathsf{S}(F)/\mathsf{S}_{2}(F)|=2^{n(g+1)}.$ This shows the optimality of the previously mentioned bounds for function field in one variables over $\mathbb{R}(\!(t)\!)$ or for hyperelliptic function fields over $\mathbb{R}(\!(t_{1})\!)\ldots(\!(t_{n})\!).$ We ask in \Cref{optimality} whether our bounds in terms of valuations may equally serve to find a bound in terms of $g$ and $n$ for function field in one variables $F/K$ of genus $g.$

\section{Preliminaries} 

For a ring $R,$ let $R^{\times}$ denote its group of invertible elements. Let $K$ be a field. For a valuation $v$ on $K,$ we denote by $\mathcal{O}_{v}$ the valuation ring associated to $v,$ by $\Gamma_{v}$ the value group $v(K^{\times}),$ which is an ordered abelian group, by $\mathfrak{m}_{v}$ the maximal ideal of $\mathcal{O}_{v}$ and by $\kappa_{v}$ the residue field $\mathcal{O}_{v}/\mathfrak{m}_{v}.$ The residue in $\kappa_{v}$ of an element $a\in\mathcal{O}_{v}$ will be denoted by $\overline{a}.$ For a valuation ring $\mathcal{O}$ of $K,$ the corresponding valuation on $K$ with valuation ring $\mathcal{O}$ and value group $K^{\times}/\mathcal{O}^{\times}$ is denoted by $v_{\mathcal{O}}.$ We call two valuations $v$ and $w$ on $K$ \textit{equivalent} if $\mathcal{O}_{v}=\mathcal{O}_{w},$ equivalently, if and only if there exists an order-isomorphism $\gamma:\Gamma_{w}\to\Gamma_{v}$ such that $v=\gamma\circ w.$ See \cite[Proposition 2.1.3]{EP}.

The \textit{rank} of an ordered abelian group $\Gamma$ is defined as the number of its proper convex subgroups. For a valuation $v$ on $K,$ we denote by $\mathsf{rk}(v)\in\mathbb{N}$ the rank of $\Gamma_{v},$ and we say that $\mathsf{rk}(v)$ is the rank of $v.$ Let $n\in\mathbb{N}.$ We call a valuation on $K$ having value group $(\mathbb{Z}^{n},\leq_{\mathsf{leq}}),$ where $\leq_{\mathsf{leq}}$ is the lexicographic order, a $\mathbb{Z}^{n}$-\emph{valuation} for short. Note that $(\mathbb{Z}^{0},\leq_{\mathsf{leq}})$ is the trivial additive group and that for a $\mathbb{Z}^{n}$-valuation $v$ on $K$ we have that $\mathsf{rk}(v)=n.$

For $i\in\mathbb{N},$ let $V_{i}(K)$ be the set of $\mathbb{Z}^{i}$-valuations on $K,$ where $V_{0}(K)$ is the trivial valuation on $K.$ We set $$V(K)=\bigcup_{i\in\mathbb{N}}V_{i}(K).$$
Let $\Omega_{n}(K)=\{\mathcal{O}_{v}\mid v\in V_{n}(K)\},$ where $\Omega_{0}(K)=K.$ We set $$\Omega(K)=\bigcup_{i\in\mathbb{N}}\Omega_{i}(K).$$
We observe that every $\mathcal{O}\in\Omega_{n}(K)$ is a valuation ring of Krull dimension $n$ by \cite[Lemma 2.3.1]{EP}. For valuation rings $\mathcal{O},\mathcal{O}'$ of a field $K$ and valuations $w,w'$ on $K$ corresponding to $\mathcal{O}$ and $\mathcal{O}',$ respectively, we say that $\mathcal{O}'$ is a \textit{coarsening of} $\mathcal{O}$ or that $\mathcal{O}$ is a \textit{refinement of} $\mathcal{O}',$ respectively, $w'$ is a coarsening of $w$ or $w$ is a refinement of $w,$ if $\mathcal{O}\subseteq \mathcal{O}'.$ Note that a field $K$ is trivially a coarsening of every valuation ring of itself. We denote by $\mathcal{O}\mathcal{O}'$ the smallest subring of $K$ containing both $\mathcal{O}$ and $\mathcal{O}'.$

Let $n,d\in\mathbb{N}$ with $d\leq n.$ We denote by $\pi_{d}:\mathbb{Z}^{n}\to\mathbb{Z}^{d}$ the projection on the first $d$ components. Note that $\pi_{d}$ is a homomorphism of ordered abelian groups with respect to the lexicographic orders on $\mathbb{Z}^{n}$ and $\mathbb{Z}^{d}.$ Dually, we denote by $\pi^{d}:\mathbb{Z}^{n}\to\mathbb{Z}^{d}$ the projection on the last $d$ components of $\mathbb{Z}^{n}.$ Note that $\pi^{d}$ is a group homomorphism, but it is not order-preserving when $d<n.$

\begin{prop}\label{coarsening} Let $n\in\mathbb{N}.$ Let $v$ be a $\mathbb{Z}^{n}$-valuation on $K.$ Let $\mathcal{O}'$ be a coarsening of $\mathcal{O}_{v}.$ Then $\mathcal{O}'=\mathcal{O}_{\pi_{r}\circ v},$ for a unique $r\leq n.$
\end{prop}
\begin{proof} First, note that $K^{\times}/\mathcal{O}_{v}^{\times}$ is an abelian group, which has an order $\leq,$ given by $x\mathcal{O}_{v}^{\times}\leq y\mathcal{O}_{v}^{\times}$ if and only if $yx^{-1}\in\mathcal{O}_{v}.$ Thus $\Gamma_{v}$ is order-isomorphic to $K^{\times}/\mathcal{O}_{v}^{\times}.$ Let $\mathcal{O}=\mathcal{O}_{v}.$ Let $\mathfrak{p}$ be the maximal ideal of $\mathcal{O}'.$ Then $\mathfrak{p}$ is a prime ideal of $\mathcal{O}$ and localizing $\mathcal{O}$ at $\mathfrak{p}$ we obtain $\mathcal{O}'=\mathcal{O}_{\mathfrak{p}}.$ Thus, we have a surjective order-preserving group homomorphism of value groups $\varphi: K^{\times}/\mathcal{O}^{\times}\to K^{\times}/\mathcal{O}_{\mathfrak{p}}^{\times}$ sending $x\mathcal{O}^{\times}$ to $x\mathcal{O}_{\mathfrak{p}}^{\times}.$ In fact, if $a\mathcal{O}^{\times}\leq b\mathcal{O}^{\times},$ then $ab^{-1}\in\mathcal{O}\subseteq \mathcal{O}_{\mathfrak{p}}$ and hence $a\mathcal{O}_{\mathfrak{p}}\leq b\mathcal{O}_{\mathfrak{p}},$ for $a,b\in K^{\times}.$ Since $K^{\times}/\mathcal{O}^{\times}\simeq \mathbb{Z}^{n},$ then the quotient group $(K^{\times}/\mathcal{O}^{\times})/\ker \varphi$ is order-isomorphic to $\mathbb{Z}^{r},$ for some $r\in\mathbb{N}$ such that $r\leq n.$ Hence $K^{\times}/\mathcal{O}_{\mathfrak{p}}\simeq \mathbb{Z}^{r}$ by the first isomorphism theorem. Let $v'=\pi_{r}\circ v.$ Since $\pi_{r}$ is order-preserving, we have that $\mathcal{O}_{v'}\in\Omega_{r}(K),$ which is a coarsening of $\mathcal{O}_{v}.$ By \cite[Lemma 2.3.1]{EP} and since $\mathbb{Z}^{n}$ has only $n$ proper convex subgroups, we have that $\mathcal{O}'=\mathcal{O}_{v'}.$
\end{proof}

Let $v\in V(K)$ and set $n=\mathsf{rk}(v).$ For $1\leq i\leq n,$ we denote by $e_{i}^{n}$ the $n$-tuple $(e_{1},\ldots,e_{n})$ such that $$e_{j}=\left\{\begin{array}{ll} 1 & \mbox{ if $j=(n+1)-i$ ,}\\ 0 &\mbox{ if $j\neq (n+1)-i$ .}\end{array}\right.$$

Note that $e_{1}^{n}$ is the minimal positive element of $\mathbb{Z}^{n}.$ An $n$-tuple $(t_{1},\ldots,t_{n})\in K^{n}$ is called a \textit{parametrical system of} $v$ if $v(t_{i})=e_{i}^{n}$ for all $1\leq i\leq n.$ For a $\mathbb{Z}$-valuation $v$ we call $t\in K$ a \textit{uniformizer of} $v$ if $v(t)=1.$

\begin{df}
Let $n\in\mathbb{N}.$ Let $v$ be a $\mathbb{Z}^{n}$-valuation on $K$ and let $\mathcal{O}'$ be a coarsening of $\mathcal{O}_{v}.$ By \Cref{coarsening} there exists a unique $r\leq n,$ such that $\mathcal{O}'=\mathcal{O}_{\pi_{r}\circ v}.$ Let $v'=\pi_{r}\circ v.$ Let $x,y\in\mathcal{O}_{v'}$ be such that $x-y\in\mathfrak{m}_{v'}$ and $x,y\notin\mathfrak{m}_{v'}.$ Since $y\in\mathcal{O}_{v'}^{\times}$ and $\mathfrak{m}_{v'}\subseteq \mathfrak{m}_{v},$ we have that $(x-y)y^{-1}=x/y-1\in\mathfrak{m}_{v},$ and hence $v(x)=v(y).$ We denote by $\overline{v}:\kappa_{v'}\to\mathbb{Z}^{n-r}\cup\{\infty\},\overline{x}\mapsto (\pi^{n-r}\circ v)(x),$ where $x$ is any lifting of $\overline{x},$ the \textit{residual valuation of} $v$ \textit{modulo} $v',$ which is well-defined by the above.  

Let $r,d\in\mathbb{N}.$ On the other hand, let $v'$ be a $\mathbb{Z}^{r}$-valuation on $K$ and let $\overline{v}$ be a $\mathbb{Z}^{d}$-valuation on $\kappa_{v'}.$ Let $(t_{1},\ldots,t_{r})\in K^{r}$ be a parametrical system of $v'.$ For $a\in K^{\times},$ there exists $a_{1},\ldots,a_{r}\in\mathbb{Z}$ such that $a=t_{1}^{a_{1}}\ldots t_{r}^{a_{r}}u,$ for some $u\in \mathcal{O}_{v'}^{\times}.$ We can define a valuation $v:K\to \mathbb{Z}^{r}\times\mathbb{Z}^{d}\cup\{\infty\},$ by $v(a)=(v'(a),\overline{v}(\overline{u})).$ The valuation $v$ is called \textit{the composition of} $v'$ and $\overline{v}$ with respect to $(t_{1},\ldots,t_{r}).$ Note that $\mathcal{O}_{v}$ is a refinement of $\mathcal{O}_{v'}.$
\end{df}

\begin{ex} Let $n\in\mathbb{N}^{+}.$ We consider $K_{n}=k(\!(t_{1})\!)\ldots(\!(t_{n})\!)$ the field of iterated Laurent series over a field $k.$ By induction on $n,$ we show that there exists a $\mathbb{Z}^{n}$-valuation $v$ on $K_{n}$ such that $(t_{1},\ldots,t_{n})\in K_{n}^{n}$ is a parametrical system of $v.$ For $n=1,$ the valuation $v_{t_{1}}$ on $K_{1},$ given by $v_{t_{1}}(\displaystyle\sum_{i=m}^{\infty}a_{i}t_{1}^{i})=m,$ when $a_{m}\in k^{\times},$ is a $\mathbb{Z}$-valuation on $K_{1}$ such that $t_{1}$ is a parametrical system $v_{t_{1}}.$ Let $n>1.$ Let $v_{t_{n}}$ be the $t_{n}$-adic valuation on $K_{n},$ that is, the $\mathbb{Z}$-valuation corresponding to the valuation ring $K_{n-1}[\![t_{n}]\!].$ By the induction hypothesis, since $\kappa_{v_{t_{n}}}=K_{n-1},$ we may consider a $\mathbb{Z}^{n-1}$-valuation $w$ on $K_{n-1}$ such that $(t_{1},\ldots,t_{n-1})\in \kappa_{v_{t_{n}}}^{n-1}$ is a parametrical system of $w.$ Let $v$ be the composition of $v_{t_{n}}$ and $w$ with respect to $t_{n}.$ Hence $v$ is a $\mathbb{Z}^{n}$-valuation on $K_{n},$ and since $v(t_{n})=(v_{t_{n}}(t_{n}),w(\overline{1}))=(1,0,\ldots,0)=e_{n}^{n},$ we have that $(t_{1},\ldots,t_{n})\in K_{n}^{n}$ is a parametrical system of $v.$ Note that different choices of uniformizers of $v_{t_{n}}$ can lead to distinct $\mathbb{Z}^{n}$-valuations on $K_{n}$ (e.g. $t_{n}t_{1}$ instead of $t_{n}$).
\end{ex}

\begin{rem}\label{inversevaluations} Let $r\in\mathbb{N}.$ Fixing a $\mathbb{Z}^{r}$-valuation $v'$ on $K,$ a parametrical system $(t_{1},\ldots,t_{r})\in K^{r}$ for $v',$ and $d\in\mathbb{N},$ the above defined assignations define mutually inverse bijections between equivalence classes of $\mathbb{Z}^{r+d}$-valuations $v,$ where $\mathcal{O}_{v}$ are refinements of $\mathcal{O}_{v'}$ and equivalence classes of $\mathbb{Z}^{d}$-valuations $\overline{v}$ on $\kappa_{v},$ such that $\kappa_{v}=\kappa_{\overline{v}}.$ See \cite[Section 2.3]{EP}.

\end{rem}

We call $W\subseteq \Omega(K)\smallsetminus\Omega_{0}(K)$ \textit{saturated} if for all $\mathcal{O}\in W$ we have $\mathcal{O}'\in W$ for every nontrivial coarsening $\mathcal{O}'$ of $\mathcal{O}.$ Let $S,S'\subseteq V(K).$ We say that $S$ and $S'$ are \textit{equivalent} if $\{\mathcal{O}_{v}\mid v\in S\}=\{\mathcal{O}_{v}\mid v\in S'\}.$ We say that $S\subseteq V(K)\smallsetminus V_{0}(K)$ is \textit{coherent} if 
\begin{itemize}
\item $S$ is a set of pairwise non-equivalent valuations, 
\item and, if $v\in S$ is a $\mathbb{Z}^{n}$-valuation, then $\pi_{i}\circ v\in S,$ for all $1\leq i\leq n.$
\end{itemize}

\begin{prop}\label{coherentset} Let $S\subseteq V(K)\smallsetminus V_{0}(K)$ be such that $W=\{\mathcal{O}_{v}\mid v\in S\}$ is finite and saturated. Then there exists a coherent set $S'\subseteq V(K)$ equivalent to $S.$  
\end{prop}
\begin{proof} Let $n=\max\{\dim \mathcal{O}_{v}\mid v\in S\}.$ We prove the statement by induction on $n.$ If $n=1,$ there is nothing to show, since every subset of $V_{1}(K)$ is a set of pairwise of non-equivalent $\mathbb{Z}$-valuations, which is trivially coherent. Assume now that $n>1.$ By the induction hypothesis, we have that for any $S\subseteq \bigcup_{1\leq i\leq n-1}V_{i}(K)$ such that $W=\{\mathcal{O}_{v}\mid v\in S\}$ is finite and saturated, there exists a coherent set $S'\subseteq V(K)$ with $W=\{\mathcal{O}_{v}\mid v\in S'\}.$ Let $S\subseteq \bigcup_{1\leq i\leq n}V_{i}(K)$ be such that $W=\{\mathcal{O}_{v}\mid v\in S\}$ is finite and saturated. Without loss of generality, we may assume that $S$ consist of a set of non-equivalent valuations. Let $S_{n-1}=S\smallsetminus V_{n}(K).$ Then, by the induction hypothesis, $S_{n-1}$ is equivalent to a coherent subset $S_{n-1}'$ of $\bigcup_{1\leq i\leq n-1}V_{i}(K).$ Let $v\in S_{n-1}'.$ Let $V_{v}=\{w\in S\mid \mathcal{O}_{w}\subsetneqq \mathcal{O}_{v}\}.$ Let $(t_{1},\ldots,t_{n-1})\in K^{n-1}$ be a parametrical system of $v.$ For $w\in V_{v},$ let $\overline{w}$ be the residual valuation of $w$ modulo $v.$  Let $w'$ be the composition of $v$ and $\overline{w},$ with respect to $(t_{1},\ldots,t_{n-1}).$ Thus, for every $w\in V_{v},$ we can define a $\mathbb{Z}^{n}$-valuation $w'$ on $K,$ which is equivalent to $w,$ by Remark \ref{inversevaluations}. Let $S(v)=\{w'\mid w\in V_{v}\}.$ Then $S(v)\cup S_{n-1}'$ is a coherent set equivalent to $V_{v}\cup S_{n-1}'.$ We obtain that $$S'=\displaystyle\bigcup_{v\in S_{n-1}'\cap V_{n-1}(K)}S(v)\cup S_{n-1}',$$ is a coherent finite set equivalent to $S.$ 
\end{proof}


Let $v,w\in V(K).$ We call $(\gamma_{v},\gamma_{w})\in\Gamma_{v}\times\Gamma_{w}$ \textit{compatible with respect to} $(v,w)$ if $\pi_{r}(\gamma_{v})=\pi_{r}(\gamma_{w}),$ where $r\in\mathbb{N}$ is such that $\mathcal{O}_{v}\mathcal{O}_{w}\in\Omega_{r}(K).$ 

Let $n\in\mathbb{N}.$ We recall that $\pi^{1}$ is the projection on the last component of $\mathbb{Z}^{n}.$ Let $S\subseteq V(K)\smallsetminus V_{0}(K)$ be a finite coherent set. We define the group homomorphism
\begin{align}\label{funcionhaciaZ}
\Phi_{S}:K^{\times}&\to \prod_{v\in S}\mathbb{Z},a\mapsto (\pi^{1}(v(a)))_{v\in S}.
\end{align}
 
\begin{prop}\label{surjectivity} Let $K$ be a field. Let $S\subseteq V(K)\smallsetminus V_{0}(K)$ be a finite coherent set. Then $\Phi_{S}$ is surjective. 
\end{prop}
\begin{proof} Let $(e_{v})_{v\in S}$ be the canonical basis of $\prod_{v\in S}\mathbb{Z}$ as a $\mathbb{Z}$-module. For every $v\in S,$ we show that there exists $x_{v}\in K^{\times}$ such that $\Phi_{S}(x_{v})=e_{v}.$ Consider $v\in S$ and let $\gamma_{v}=e_{1}^{\mathsf{rk}(v)}.$ For $w\in S$ a refinement of $v,$ let $\gamma_{w}=e_{\mathsf{rk}(w)-\mathsf{rk}(v)+1}^{\mathsf{rk}(w)}$ in $\Gamma_{w},$ otherwise $\gamma_{w}=0.$ We claim that for every $w,w'\in S,$ the pair $(\gamma_{w},\gamma_{w'})$ in $\Gamma_{w}\times \Gamma_{w'},$ is compatible with respect to $(w,w')$. Let $w,w'\in S.$ If $w,w'$ are both not refinements of $v,$ then $(\gamma_{w},\gamma_{w'})$ are trivially compatible. Assume $w$ is a refinement of $v$ and $w'$ is not a refinement of $v.$ It follows from \Cref{coarsening} that $\mathcal{O}_{w}\subseteq\mathcal{O}_{v}\subsetneqq\mathcal{O}_{w}\mathcal{O}_{w'}\subseteq K,$ because every valuation ring has a unique coarsening of a fixed rank. Let $d\in\mathbb{N}$ be such that $\mathcal{O}_{w}\mathcal{O}_{w'}\in\Omega_{d}(K).$ Then $d<\mathsf{rk}(v)\leq \mathsf{rk}(w)$ and hence $\pi_{d}(\gamma_{w})=0=\pi_{d}(\gamma_{w'})$ in $\mathbb{Z}^{d}.$ Finally, we assume that $w$ and $w'$ are refinements of $v.$ Let $d\in\mathbb{N}$ be such that $\mathcal{O}_{w}\mathcal{O}_{w'}\in\Omega_{d}(K).$ Then $\mathcal{O}_{w}\mathcal{O}_{w'}\subseteq \mathcal{O}_{v}$ and hence we have that $\mathsf{rk}(v)\leq d< \mathsf{rk}(w),\mathsf{rk}(v)\leq d< \mathsf{rk}(w').$ Thus, we have $\pi_{d}(\gamma_{w})=\pi_{d}(\gamma_{w'}),$ that is, $(\gamma_{w},\gamma_{w'})$ are compatible. Therefore, for every $w,w'\in S,$ the pair $(\gamma_{w},\gamma_{w'})\in\Gamma_{w}\times\Gamma_{w'}$ is compatible. By \cite[Theorem 5]{Rib57} there exists $x_{v}\in K^{\times}$ such that $w(x_{v})=\gamma_{w}$ for all $w\in W.$ Hence $\pi^{1}(v(x_{v}))=\pi^{1}(\gamma_{v})=1,$ and $\pi^{1}(w(x_{v}))=0$ for all $w\in S\setminus \{v\}.$ Therefore $\Phi_{S}(x_{v})=e_{v}.$ Since $v\in S$ was arbitrarily chosen, we conclude that $\Phi_{S}$ is surjective.  
\end{proof}

\begin{prop}\label{unicidad} Let $n\in\mathbb{N}.$ Assume that $K$ carries a henselian $\mathbb{Z}^{n}$-valuation. Then $|\Omega_{i}(K)|=1$ for all $1\leq i\leq n.$ 
\end{prop}
\begin{proof} We prove this by induction on $n.$ For $n=1,$ follows from \cite[Proposition 2.2]{BGVG}. We assume that $n>1.$ Let $v_{n-1}=\pi_{n-1}\circ v.$ Then $v_{n-1}$ is a henselian $\mathbb{Z}^{n-1}$-valuation on $K,$ which is a coarsening of $v,$ by \cite[Corollary 4.1.4]{EP} and \Cref{coarsening}. By the induction hypothesis $|\Omega_{i}(K)|=1,$ for all $1\leq i\leq n-1.$ Let $\overline{v}$ be the residual valuation of $v$ modulo $v_{n-1}.$ Then $\overline{v}$ is a henselian $\mathbb{Z}$-valuation on $\kappa_{v_{n-1}},$ by \cite[Corollary 4.1.4]{EP}, and it is the unique $\mathbb{Z}$-valuation on $\kappa_{v_{n-1}},$ by \cite[Proposition 2.2]{BGVG}. Therefore, by \Cref{inversevaluations}, we have that $\Omega_{n}(K)=\{\mathcal{O}_{v}\}$ and hence $|\Omega_{n}(K)|=1.$  
\end{proof}

By a function field in one variable $F/K$ we mean a finitely generated field extension of transcendence degree one. 

\begin{lem}\label{algebraicfunctionfield} Let $F/K$ be a function field in one variable and $r\in\mathbb{N}.$ Let $w$ be a valuation on $F$ such that $\kappa_{w}/\kappa_{w|_{K}}$ is transcendental. Then $\mathsf{rk}(w)=\mathsf{rk}(w|_{K})$ and $\kappa_{w}/\kappa_{w|_{K}}$ is a function field in one variable.
\end{lem}
\begin{proof} Let $v=w|_{K}.$ It follows from \cite[Proposition 3.4.1, Theorem 3.4.3]{EP} that $\mathsf{rk}(w)=\mathsf{rk}(v).$ We claim that $\kappa_{w}/\kappa_{v}$ is a function field in one variable. Let $\alpha\in \kappa_{w}$ be a transcendental element over $\kappa_{v}$ and let $\theta\in\mathcal{O}^{\times}$ be such that $\overline{\theta}=\alpha.$ It follows by \cite[Theorem 3.2.4]{EP} that $K(\theta)/K$ is transcendental. Let $w':=w|_{K(\theta)}.$ It follows from \cite[Corollary 2.2.2]{EP} that $w'$ is the Gauss extension of $v$ with respect to $\theta.$ In particular $\kappa_{w'}=\kappa_{v}(\alpha).$ Since $F/K$ is a function field in one variable, we have that $F/K(\theta)$ is finite, then $\kappa_{w}/\kappa_{v}(\alpha)$ is a finite field extension and hence $\kappa_{w}/\kappa_{v}$ is a function field in one variable. 
\end{proof}

We fix a valued field $(K,v)$ and $F/K$ a function field in one variable. An extension $w$ of $v$ to $F$ is called \emph{residually transcendental} if $\kappa_{w}/\kappa_{v}$ is transcendental. We say that $F/K$ is \emph{ruled} if $F$ is a rational function field over some finite extension of $K.$ 

\begin{prop}\label{Ohm} Let $v$ be a valuation on $K.$ Let $F/K$ be a ruled extension. Let $w$ be a residually transcendental extension of $v.$ Then $\kappa_{w}/\kappa_{v}$ is ruled.	
\end{prop}	
\begin{proof} See \cite[Theorem 3.3]{ohm}.
\end{proof}

We define a set of equivalence classes of valuation extensions of coarsenings of $v$ to $F$ in the following way. For $i\in\mathbb{N},$ let $\Omega_{i}(F/v)$ denote the set of valuation rings $\mathcal{O}\in\Omega_{i}(F)$ such that $v_{\mathcal{O}}$ is a residually transcendental extension of a coarsening of $v.$ We set $$\Omega(F/v)=\bigcup_{i\in\mathbb{N}}\Omega_{i}(F/v).$$

\begin{prop}\label{discreteresidually} Let $n\in\mathbb{N}$ and let $v\in V_{n}(K).$ Let $F/K$ be a function field in one variable. Let $\mathcal{O}\in\Omega(F/v).$ Then $v_{\mathcal{O}}$ is a $\mathbb{Z}^{r}$-valuation, for some $r\leq n.$ 
\end{prop}	
\begin{proof} This follows directly from \Cref{coarsening} and Lemma \ref{algebraicfunctionfield}.
\end{proof}

For $i\in\mathbb{N}^{+},$ we set $$\Omega_{i}^{*}(F/v)=\{\mathcal{O}\in\Omega_{i}(F/v)\mid\:\kappa_{\mathcal{O}}/\kappa_{\mathcal{O}\cap K}\:\text{is nonruled}\:\}.$$ We set $$\Omega(F/v)=\bigcup_{i\in\mathbb{N}}\Omega_{i}^{*}(F/v).$$

\begin{lem}\label{nonruledcontained} Let $F/K$ be a function field in one variable. Then $\Omega^{*}(F/v)$ is saturated. 
\end{lem}
\begin{proof} Let $w$ be a valuation on $F$ such that $\mathcal{O}_{w}\in\Omega^{*}(F/v).$ Let $w'$ be a coarsening of $w.$ We claim that $\mathcal{O}_{w'}\in\Omega^{*}(F/v).$ It follows by \cite[Theorem 3.2.4]{EP} that $\kappa_{w'}/\kappa_{w'|_{K}}$ is transcendental, and hence is a function field in one variable, by Lemma \ref{algebraicfunctionfield}. Let $\overline{w}$ be the residual valuation of $w$ modulo $w'.$ Let $\nu=w|_{K}.$ Note that $w'|_{K}$ is a coarsening of $\nu.$ Let $\overline{\nu}$ be the residual valuation of $\nu$ modulo $w'|_{K}.$ If $\kappa_{w'}/\kappa_{w'|_{K}}$ were ruled, then $\kappa_{\overline{w}}/\kappa_{\overline{\nu}}$ would be ruled, by Proposition \ref{Ohm}, and hence $\kappa_{w}/\kappa_{\nu}$ would be ruled because $\kappa_{w}=\kappa_{\overline{w}},\kappa_{\nu}=\kappa_{\overline{\nu}},$ which is a contradiction. Therefore $\mathcal{O}_{w'}\in\Omega^{*}(F/v).$ 
\end{proof}

Let $F/K$ be a function field in one variable. It is natural to wonder whether the set $\Omega^{*}(F/v)$ is finite. We shall give a positive answer in \Cref{finiteset} under the assumption that $v\in V(K).$ Assuming that $v$ is a $\mathbb{Z}$-valuation on $K,$ it was shown in \cite[Theorem 5.3]{BG20} by K. Becher and D. Grimm $|\Omega^{*}(F/v)|\leq g+1,$ where $g$ is the genus of $F/K.$ We will show in \Cref{exoptimality} that the above bound is optimal. 

\begin{thm}\label{finiteset} Let $n$ be a positive integer. Assume that $K$ carries a $\mathbb{Z}^{n}$-valuation $v$ such that $\kappa_{v}$ is perfect. Let $F/K$ be a function field in one variable. Then $\bigcup_{1\leq i\leq n}\Omega_{i}^{*}(F/v)$ is finite.
\end{thm}
\begin{proof} We prove the statement by induction on $n.$ If $n=1,$ then $v$ is a $\mathbb{Z}$-valuation on $K,$ and it follows by \cite[Theorem 5.3]{BG20} that $\Omega_{1}^{*}(F/v)$ is finite. Assume now that $n>1.$ By the induction hypothesis, for any positive integer $s<n,$ for any field $L$ carrying a $\mathbb{Z}^{s}$-valuation $v'$ and for any function field in one variable $E/L,$ the set $\bigcup_{1\leq i\leq s}\Omega_{i}^{*}(E/v')$ is finite. Let $v_{1}=\pi_{1}\circ v$ and let $\overline{v}$ be the residual valuation of $v$ modulo $v_{1}.$ Then $\overline{v}$ is a $\mathbb{Z}^{n-1}$-valuation on $\kappa_{v_{1}},$ by \Cref{remark}. Let $r\in\{1,\ldots,n\}.$ We claim that $\Omega_{r}^{*}(F/v)$ is finite. If $r=1,$ then this follows by \cite[Theorem 5.3]{BG20}. Assume $r>1.$ Let $w$ be a valuation on $F$ such that $\mathcal{O}_{w}\in\Omega_{r}^{*}(F/v).$ Let $w_{1}=\pi_{1}\circ w.$ It follows by Lemma \ref{nonruledcontained} and by Lemma \ref{algebraicfunctionfield} that $\mathcal{O}_{w_{1}}\in\Omega_{1}^{*}(F/v)$ and that $\mathcal{O}_{w_{1}}$ is an extension of $\mathcal{O}_{v_{1}}.$ Since $\kappa_{w_{1}}/\kappa_{v_{1}}$ is a function field in one variable and $\overline{v}$ is a $\mathbb{Z}^{n-1}$-valuation on $\kappa_{v_{1}},$ we have that $\Omega_{r-1}^{*}(\kappa_{w_{1}}/\overline{v})$ is finite. Furthermore, we have that $\mathcal{O}_{w}$ is determined by the induce valuation ring $\mathcal{O}_{\overline{w}}\in\Omega_{r-1}^{*}(\kappa_{w_{1}}/\overline{v}),$ where $\overline{w}$ is the residual valuation of $w$ modulo $w_{1},$ by Remark \ref{inversevaluations}. Hence, we have that $|\Omega_{r}^{*}(F/v)|=\sum_{\mathcal{O}\in\Omega_{1}^{*}(F/v)}|\Omega_{r-1}^{*}(\kappa_{\mathcal{O}}/\overline{v})|,$ where $\overline{v}$ is the residual valuation of $v$ modulo $v_{\mathcal{O}}.$ Since for every $\mathbb{Z}$-valuation $\nu$ on $F$ the set $\Omega_{r-1}(\kappa_{\nu}/\overline{v})$ is finite, we have that the set $\Omega_{r}^{*}(F/v)$ is finite, and since $r$ was arbitrarily taken, we obtain the statement. 
\end{proof}


Let $v$ be a $\mathbb{Z}$-valuation on $K.$ We denote by $(K^{v},\hat{v})$ the completion of $(K,v).$  

\begin{lem}\label{extension} Let $v$ be a henselian $\mathbb{Z}^{n}$-valuation on $K,$ and let $v_{1}=\pi_{1}\circ v.$ Let $(K^{v_{1}},\hat{v}_{1})$ be the completion of $(K,v_{1}).$ Then $v$ has a unique unramified extension $v'$ to $K^{v_{1}},$ and this extension is a henselian $\mathbb{Z}^{n}$-valuation with $\kappa_{v'}=\kappa_{v}.$ 
\end{lem}
\begin{proof} We observe that $\kappa_{v_{1}}=\kappa_{\hat{v}_{1}}$ by \cite[Theorem 1.3.4]{EP}. Let $\overline{v}$ be the residual valuation of $v$ modulo $v_{1},$ and let $t$ be an uniformizer of $\hat{v}_{1}.$ We can consider the composition $v'$ of $\hat{v}_{1}$ and $\overline{v}$ with respect to $t.$ Note that $\mathcal{O}_{v'}\cap K=\mathcal{O}_{v},$ and using the fact that $\overline{v}$ and $\hat{v}_{1}$ are henselian $\mathbb{Z}^{n-1}$ and $\mathbb{Z}$-valuation on $\kappa_{v_{1}}$ and $K^{v_{1}},$ respectively, we have that $v'$ is a henselian $\mathbb{Z}^{n}$-valuation on $K^{v_{1}}$ by \cite[Corollary 4.1.4]{EP}, and with $\kappa_{v'}=\kappa_{\overline{v}}=\kappa_{v},$ by \Cref{inversevaluations}.
\end{proof}


Let $F/K, K'/K$ be two field extensions such that $K$ is separably algebraically closed in $K'.$ Then $F\otimes_{K}K'$ is a domain, see \cite[Corollary 1, pag. 203]{Jac64}. We denote by $FK'$ the fraction field of $F\otimes_{K}K'$ and we call it \textit{the compositum of} $F$ \textit{and} $K'$ \textit{over} $K.$ Note that $FK'$ is an extension of $F$ and of $K'.$ 

A field $K$ that carries a henselian $\mathbb{Z}$-valuation $v$ is separably algebraically closed in its completion $K^{v},$ by \cite[Theorem 32.19]{Wa89}, hence the compositum $FK^{v}$ over $K$ exists for any field extension $F/K.$ In particular, if $F/K$ is a function field in one variable, then $FK^{v}/K^{v}$ is a function field in one variable.

\begin{prop}[D. Harbater, J. Hartmann, D. Krashen]\label{HHK09} Assume that $K$ carries a henselian $\mathbb{Z}$-valuation $v$ such that $\mathcal{O}_{v}$ is excellent. Let $F/K$ be a function field in one variable, and let $E$ be the compositum of $F$ and $K^{v}$ over $K.$ Let $\varphi$ be a quadratic form over $F$ of $\dim\varphi\geq 3.$ If $\varphi$ is isotropic over $E,$ then $\varphi$ is isotropic over $F.$ 
\end{prop}
\begin{proof} See \cite[Lemma 4.11]{HHK09}.
\end{proof}

\begin{cor}\label{corolariocota} Assume that $K$ carries a henselian $\mathbb{Z}$-valuation $v$ such that $\mathcal{O}_{v}$ is excellent. Let $F/K$ be a function field in one variable, and let $E$ be the compositum of $F$ and $K^{v}$ over $K.$ Then $$F^{\times}\cap \mathsf{S}_{k}(E)=\mathsf{S}_{k}(F),$$ for any $k\geq 2.$ In particular $p(F)\leq p(E).$ 
\end{cor}
\begin{proof} Let $k\geq 2.$ Let $\sigma\in F^{\times}\cap \mathsf{S}_{k}(E).$ Let $\varphi=k\times\langle 1\rangle\perp \langle-\sigma\rangle.$  Since $\varphi$ is isotropic over $E,$ then $\varphi$ is isotropic over $F$ by \Cref{HHK09}, hence $\sigma\in \mathsf{S}_{k}(F).$ The second statement follows trivially. 
\end{proof}

\section{The group of nonzero sums of squares} 

\label{Section HPF}



Let $K$ be a field. We define $$s(K)=\inf\{d\in\mathbb{N}\mid -1\in S_{d}(K)\}\in\mathbb{N}\cup\{\infty\},$$ the \textit{level} of $K.$ If $s(K)<\infty,$ we say that $K$ is \emph{nonreal}, otherwise we say that $K$ is real.\\\\
For a valued field $(K,v),$ a function field in one variable $F/K,$ and $r\in\mathbb{N},$ we define $$\mathcal{X}^{r}(F/v)=\{\mathcal{O}\in\Omega^{*}(F/v)\mid 2^{r}\leq s(\kappa_{\mathcal{O}})<\infty\:\},$$ and 
\begin{equation*} 
\begin{split}
\mathcal{E}^{r}(F/v) & = \{x\in F^{\times}\mid x\in\mathcal{O}^{\times}F^{\times 2}\:\mbox{for all}\:\mathcal{O}\in\mathcal{X}^{r}(F/v)\} \\
 & = \{x\in F^{\times}\mid v_{\mathcal{O}}(x)\in 2\Gamma_{v_{\mathcal{O}}}\:\mbox{for all}\:\mathcal{O}\in\mathcal{X}^{r}(F/v)\}.
\end{split}
\end{equation*}

\begin{lem}\label{aresidual} Let $K$ be a field and $v\in V(K).$ Let $F/K$ be a function field in one variable. Let $\mathcal{O}\in \Omega_{1}^{*}(F/v).$ Let $a\in F^{\times}$ and $r\in\mathbb{N}.$ If $a\in\mathcal{E}^{r}(F/v)\cap\mathcal{O}^{\times},$ then $\overline{a}\in\mathcal{E}^{r}(\kappa_{\mathcal{O}}/\overline{v}),$ where $\overline{v}$ is the residual valuation of $v$ modulo $v_{\mathcal{O}}|_{K}.$ 
\end{lem}
\begin{proof} Let $v_{1}=v_{\mathcal{O}}|_{K}.$ It follows by Lemma \ref{algebraicfunctionfield} and Proposition \ref{discreteresidually} that $v_{1}$ is equivalent to a $\mathbb{Z}$-valuation. Let $\overline{v}$ be the residual valuation of $v$ modulo $v_{1}.$ Let $a\in\mathcal{E}^{r}(F/v)\cap\mathcal{O}^{\times}.$ Let $\mathcal{O}'\in\mathcal{X}^{r}(\kappa_{\mathcal{O}}/\overline{v}).$ We need to show that $\overline{a}\in\mathcal{O}'\kappa_{\mathcal{O}}^{\times 2}.$ Let $v'=v_{\mathcal{O}'}|_{\kappa_{v_{1}}}.$ Let $\pi,\pi'$ be uniformizers of $v_{\mathcal{O}}$ and $v_{1}$ respectively. Let $w$ be a composition of $v_{\mathcal{O}}$ and $v_{\mathcal{O}'}$ with respect to $\pi,$ and let $\nu$ be a composition of $v_{1}$ and $v'$ with respect to $\pi'.$ Since $\kappa_{\mathcal{O}'}=\kappa_{w}$ and $\kappa_{v'}=\kappa_{\nu}$ by Remark \ref{inversevaluations}, we have that $\kappa_{w}/\kappa_{\nu}$ is a function field in one variable with $2^{r}\leq s(\kappa_{w})<\infty$ and $\mathcal{O}_{v}\subseteq \mathcal{O}_{w|_{K}}=\mathcal{O}_{\nu},$ that is $\mathcal{O}_{w}\in\mathcal{X}^{r}(F/v).$ Then $a\in\mathcal{O}_{w}^{\times}F^{\times 2}\cap\mathcal{O}^{\times},$ because $a\in\mathcal{E}^{r}(F/v).$ This implies that $\overline{a}\in \mathcal{O}'^{\times}\kappa_{\mathcal{O}}^{\times 2},$ because $(0,v_{\mathcal{O'}}(\overline{a}))=w(a)\in 2\Gamma_{w}.$
\end{proof}

We say that a valuation $v$ on $K$ is real or nonreal, respectively, if $\kappa_{v}$ has the corresponding property. We say that a valuation $v$ on $K$ is nondyadic if $\mathsf{char}(\kappa_{v})\neq 2.$ We recall that an ordered abelian group that admits a minimal positive element is called \emph{discrete}.

\begin{lem}\label{nivelpythagoras} Let $v$ be a nonreal and nondyadic valuation on $K$ such that $\Gamma_{v}$ is discrete. Then $s(\kappa_{v})<p(K).$ 
\end{lem}
\begin{proof} Let $d\in\mathbb{N}$ be such that $d=s(\kappa_{v}).$ There exist $f\in\mathfrak{m}_{v},x_{1},\ldots,x_{d}\in K$ such that $f=1+x_{1}^{2}+\cdots +x_{d}^{2}.$ Let $b=(1-\frac{f}{2})^{2}+x_{1}^{2}+\cdots +x_{d}^{2}=\frac{f^{2}}{4}.$ Let $\gamma\in\Gamma_{v}$ be the minimal positive element and let $z\in K$ be such that $v(z)=\gamma.$ Note that $\gamma\notin 2\Gamma_{v}.$ Hence $0<v(z)<2v(f)=v(b).$ Let $\sigma=(z-(1-\frac{f}{2}))^{2}+x_{1}^{2}+\cdots +x_{d}^{2}.$ Since $\sigma=z(z-2(1-\frac{f}{2}))+b,\:0<v(z)<v(b),$ and since $v(2-f)=0<v(z),$ we have that $v(\sigma)=\gamma.$ It follows by Lemma \cite[Lemma 4.1]{BGVG} that $\sigma\notin \mathsf{S}_{d}(K),$ and hence $s(\kappa_{v})<p(K).$ 
\end{proof}

\begin{lem}\label{contencion} Let $r\in\mathbb{N}^{+}.$ Let $K$ be a field and $v$ a real henselian valuation in $V(K)$ such that $p(\kappa_{v}(X))\leq 2^{r}.$ Let $v_{1}=\pi_{1}\circ v.$ Let $\overline{v}$ be the residual valuation of $v$ modulo $v_{1}.$ Let $(K^{v_{1}},\hat{v}_{1})$ be the completion of $(K,v_{1}).$ Let $v'$ be a composition of $\hat{v}_{1}$ and $\overline{v}$ with respect to a uniformizer of $\hat{v}_{1}.$ Let $F/K$ be a regular function field in one variable. Let $E$ be the compositum of $F$ and $K^{v_{1}}$ over $K.$ Then $$\mathsf{S}(F)\cap\mathcal{E}^{r}(F/v)\subseteq \mathsf{S}(E)\cap\mathcal{E}^{r}(E/v').$$ 
\end{lem}
\begin{proof} Let $\mathcal{O}\in\mathcal{X}^{r}(E/v').$ Let $\mathcal{O}_{F}=\mathcal{O}\cap F$ and $\mathcal{O}_{K}=\mathcal{O}\cap K.$ We claim that $\mathcal{O}_{F}\in\mathcal{X}^{r}(F/v)$ or $\kappa_{\mathcal{O}_{F}}$ is real. By definition we have that $\mathcal{O}_{v'}\subseteq \mathcal{O}\cap K^{v_{1}},$ and since $v$ is henselian, we have that $\mathcal{O}\cap K^{v_{1}}\subseteq\mathcal{O}_{\hat{v}_{1}}\subseteq K^{v_{1}},$ by \Cref{unicidad}. Hence $\mathcal{O}_{v}\subseteq \mathcal{O}_{K}\subseteq \mathcal{O}_{v_{1}}\subseteq K.$ We denote by $\nu$ the residual valuation of $v$ modulo $v_{\mathcal{O}_{K}}.$ We recall that $\nu$ is a henselian valuation on $\kappa_{\mathcal{O}_{K}}$ such that $\kappa_{\nu}=\kappa_{v},$ by \cite[Corollary 4.1.4]{EP}. Since $p(\kappa_{v}(X))\leq 2^{r},$ it follows by \cite[Theorem 3.5]{BV09} that $s(L)\leq 2^{r-1}$ for every finite nonreal extension $L/\kappa_{v}.$ Let $L'/\kappa_{\mathcal{O}_{K}}$ be a finite nonreal extension. Since $\nu$ is henselian, any extension of $\nu$ to $L'$ is again henselian, and hence we obtain that $s(L')\leq 2^{r-1}.$ Since $\kappa_{\mathcal{O}_{F}}\subseteq \kappa_{\mathcal{O}},$ we have that $s(\kappa_{\mathcal{O}_{F}})\geq s(\kappa_{\mathcal{O}})\geq 2^{r}$ and hence the extension $\kappa_{\mathcal{O}_{F}}/\kappa_{\mathcal{O}_{K}}$ cannot be algebraic. Hence $\kappa_{\mathcal{O}_{F}}/\kappa_{\mathcal{O}_{K}}$ is transcendental. If $\kappa_{\mathcal{O}_{F}}$ is nonreal, since $s(\kappa_{\mathcal{O}_{F}})\geq 2^{r},$ then $2^{r}<p(\kappa_{\mathcal{O}_{F}}),$ by \Cref{nivelpythagoras}, because $v_{\mathcal{O}|_{F}}$ is equivalent to a valuation in $V(F),$ by Proposition \ref{discreteresidually}. This implies that $\kappa_{\mathcal{O}_{F}}/\kappa_{\mathcal{O}_{K}}$ cannot be ruled, by \cite[Theorem 3.5]{BV09}, whenever $\kappa_{\mathcal{O}_{F}}$ is nonreal. Hence $\mathcal{O}_{F}\in \mathcal{X}^{r}(F/v)$ because $\mathcal{O}_{v}\subseteq \mathcal{O}_{K}.$ We conclude that $s(\kappa_{\mathcal{O}_{F}})\geq 2^{r}.$ 
	
Let $a\in \mathsf{S}(F)\cap\mathcal{E}^{r}(F/v).$ 
Hence $a\in \mathcal{O}_{F}^{\times}F^{\times 2}$ because  $\mathcal{O}_{F}\in\mathcal{X}^{r}(F/v)$ or $\kappa_{\mathcal{O}_{F}}$ is real, and the latter follows by \cite[Lemma 4.1]{BGVG}. Therefore $a\in\mathcal{O}_{F}^{\times}F^{\times 2}\subseteq\mathcal{O}^{\times}E^{\times 2}.$ Since $\mathcal{O}$ was arbitrarily chosen, we conclude that $a\in \mathsf{S}(E)\cap \mathcal{E}^{r}(E/v').$  
\end{proof}

For a field $K,$ we set 
$$p'(K)=\left\{\begin{array}{ll} p(K) & \mbox{ if $K$ is real,}\\ s(K)+1 &\mbox{ if $K$ is nonreal.}\end{array}\right.$$

\begin{lem}\label{sumasdecuadradosalgebraico} Let $K$ be a perfect field. Let $L/K$ be an algebraic extension. Then $$p'(L)\leq p(K(X)).$$
\end{lem}
\begin{proof} If the extension $L/K$ is finite, then the result follows from \cite[Lemma 6.3]{BGVG}. Assume that $L/K$ is algebraic. If $p(K(X))=\infty,$ then the inequality is trivially satisfied. Assume $p(K(X))<\infty.$ Let $r=p(K(X)),$ and let $\sigma\in\mathsf{S}_{r+1}(L).$ Let $x_{0},\ldots,x_{r}\in L$ be such that $\sigma=x_{0}^{2}+\cdots +x_{r}^{2},$ and let $K'=K(x_{0},\ldots,x_{r}).$ Hence $K'/K$ is finite. Then $p(K')\leq r,$ by \cite[Lemma 6.3]{BGVG}, and since $K'\subseteq L,$ we have that $\sigma\in\mathsf{S}_{r}(L)$ because $\mathsf{S}_{r}(K')\subseteq \mathsf{S}_{r}(L).$ Therefore $p(L)\leq r.$
\end{proof}

\begin{prop}\label{doscuadrados} Let $r\in\mathbb{N}^{+}.$ Let $K$ be a field and $v$ a real henselian valuation in $V(K)$ such that $p(\kappa_{v}(X))\leq 2^{r}.$ Let $F/K$ be a regular function field in one variable. Then $$\mathsf{S}(F)\cap \mathcal{E}_{r}(F/v)=\mathsf{S}_{2^{r}}(F).$$  
\end{prop}
\begin{proof} It follows from \cite[Lemma 4.1]{BGVG} that $\mathsf{S}_{2^{r}}(F)\subseteq \mathsf{S}(F)\cap \mathcal{E}^{r}(F/v).$ Let $n=\mathsf{rk}(v).$ Let us now show by induction over $n$ that $\mathsf{S}(F)\cap \mathcal{E}_{r}(F/v)\subseteq \mathsf{S}_{2^{r}}(F).$ For $n=0,$ it follows trivially from the assumption.   

Let now $n>0.$ Let $v_{1}=\pi_{1}\circ v.$ Let $(K^{v_{1}},\hat{v}_{1})$ be the completion of $(K,v_{1}).$ Let $v'$ be a composition of $\hat{v}_{1}$ and $\overline{v}$ with respect to a uniformizer of $\hat{v}_{1}.$ It follows by Lemma \ref{extension} that $v'$ is a henselian valuation of rank $n$ on $K^{v_{1}}$ such that $\kappa_{v}=\kappa_{v'}$ and $\mathcal{O}_{v'}\cap K=\mathcal{O}_{v}.$

Let $E$ be the compositum of $F$ and $K^{v_{1}}$ over $K.$ Let $\sigma\in \mathsf{S}(F)\cap \mathcal{E}^{r}(F/v).$ Then $\sigma\in \mathsf{S}(E)\cap \mathcal{E}^{r}(E/v'),$ by Lemma \ref{contencion}. Let $\varphi$ be the quadratic form $2^{r}\times \left<1\right>\perp\left<-\sigma\right>$ over $F.$ We first show that $\varphi$ is isotropic over $E.$

Let $w\in V_{1}(E).$ We claim that $p(E^{w})\leq 2^{r}.$ If $w|_{K^{v_{1}}}$ is trivial, then $\kappa_{w}/K^{v_{1}}$ is a finite field extension, and it follows that $p(E^{w})\leq p'(\kappa_{w}),$ by \cite[Proposition 4.3]{BGVG} and that $p'(\kappa_{w})\leq 2^{r}$ by \cite[Theorem 3.5]{BV09}. Let us assume now that $\mathcal{O}_{w}\cap K^{v_{1}}=\mathcal{O}_{\hat{v}_{1}}$ and $\kappa_{w}/\kappa_{\hat{v}_{1}}$ is algebraic. Since the residual valuation $\overline{v'}$ of $v'$ modulo $\hat{v}_{1}$ is a henselian valuation on $\kappa_{\hat{v}_{1}}$ such that $p(\kappa_{\overline{v'}}(X))\leq 2^{r},$ it follows by \cite[Theorem 3.5]{BV09} and \Cref{sumasdecuadradosalgebraico} that $p'(\kappa_{w})\leq p(\kappa_{\hat{v}_{1}}(X))\leq 2^{r}.$ Since $p(E^{w})\leq p'(\kappa_{w})\leq 2^{r}$ by Lemma \ref{sumasdecuadradosalgebraico} and by \cite[Proposition 4.3]{BGVG}, we have $p(E^{w})\leq 2^{r}.$ Therefore $\varphi$ is isotropic over $E^{w}$ in both cases. 

Let us assume now that $\mathcal{O}_{w}\cap K^{v_{1}}=\mathcal{O}_{\hat{v}_{1}}$ and $\kappa_{w}/\kappa_{\hat{v}_{1}}$ is a function field in one variable. We claim that $\varphi$ is isotropic over $E^{w}.$ If $s(\kappa_{w})\leq 2^{r-1},$ since $s(E^{w})=s(\kappa_{w}),$ we have that $\varphi$ is isotropic over $E^{w}.$ Let $d\in\mathbb{N},x_{1},\ldots,x_{d}\in F$ be such that $\sigma=x_{1}^{2}+\cdots +x_{d}^{2}.$ If $\kappa_{w}$ is real, then $w(\sigma)=\min\{2w(x_{1}),\ldots,2w(x_{d})\},$ by \cite[Lemma 4.1]{BGVG}. If $2^{r}\leq s(\kappa_{w})<\infty,$ then $2^{r}<p(\kappa_{w})$ by \Cref{nivelpythagoras}. Hence $\kappa_{w}/\kappa_{\hat{v}_{1}}$ cannot be ruled by \cite[Theorem 3.5]{BV09}, which implies that $\mathcal{O}_{w}\in\mathcal{X}^{r}(E/v').$ In any case $w(\sigma)\in 2\Gamma_{w},$ because $\sigma\in\mathcal{E}^{r}(E/v'),$ and thus $w(\sigma)=2w(y)$ for some $y\in E.$ Let $\tau=\sigma y^{-2}.$ Then $\overline{\tau}\in \mathsf{S}(\kappa_{w})$ and, if $\varphi'=2^{r}\times \left< 1\right>\perp\left<-\tau\right>$ is isotropic over $E^{w},$ then $\varphi$ is isotropic over $E^{w}.$ By the induction hypothesis, we have that $\mathsf{S}_{2^{r}}(\kappa_{w})=\mathsf{S}(\kappa_{w})\cap\mathcal{E}^{r}(\kappa_{w}/\overline{v'}).$ Since $\tau\in\mathcal{E}^{r}(E/v'),$ it follows by Lemma \ref{aresidual} that $\overline{\tau}\in\mathcal{E}^{r}(\kappa_{w}/\overline{v'}),$ and hence $\overline{\tau}\in \mathsf{S}_{2^{r}}(\kappa_{w}).$ Hence $\overline{\varphi_{r}'}$ is isotropic over $\kappa_{w},$ and then $\varphi'$ is isotropic over $E^{w},$ because the extension $\hat{w}$ of $w$ to $E^{w}$ is henselian, whereby $\varphi$ is isotropic over $E^{w}.$

By \cite[Theorem 6.1]{GR16}, $\varphi$ is isotropic over $E$ if and only if $\varphi$ is isotropic over $E^{w}$ for every rank one valuation $w$ on $E$ such that $w|_{K^{v_{1}}}$ is trivial or $\mathcal{O}_{w}\cap K^{v_{1}}=\mathcal{O}_{\hat{v}_{1}}.$ If $\mathcal{O}_{w}\cap K^{v_{1}}=\mathcal{O}_{\hat{v}_{1}}$ then it follows by Lemma \ref{algebraicfunctionfield} that, either $\kappa_{w}/\kappa_{\hat{v}_{1}}$ is algebraic or $\kappa_{w}/\kappa_{\hat{v}_{1}}$ is a function field in one variable. By the above, in any case $\varphi$ is isotropic over $E,$ and it follows from \Cref{corolariocota} that $\varphi$ is isotropic over $F,$ whereby $\sigma\in \mathsf{S}_{2^{r}}(F).$ 
\end{proof}

For a field $K$ and $r\in\mathbb{N},$ we set $G_{r}(K)=\mathsf{S}(K)/\mathsf{S}_{2^{r}}(K).$ 

\begin{thm}\label{induccionsumas} Let $r\in\mathbb{N}^{+}$ and $n\in\mathbb{N}.$ Let $K$ be a field carrying a real henselian $\mathbb{Z}^{n}$-valuation $v$ such that $p(\kappa_{v}(X))\leq 2^{r}.$ Let $F/K$ be a function field in one variable. Then $$|G_{r}(F)|=2^{|\mathcal{X}^{r}(F/v)|}.$$ In particular $|G_{r}(F)|$ is finite.  

\end{thm} 
\begin{proof} It follows by \Cref{finiteset} that we may choose a finite and saturated set $W\subseteq \Omega(F)$ such that $\mathcal{X}^{r}(F/v)\subseteq W,$ and it follows by Proposition \ref{coherentset} that we can choose a coherent subset $S'\subseteq V(F)$ such that $W=\{\mathcal{O}_{w}\mid w\in S'\}.$ We consider $S=\{w\in S'\mid \:\mathcal{O}_{w}\in\mathcal{X}^{r}(F/v)\}.$ Let $\Phi:S(F)\to \prod_{w\in S}\mathbb{Z}/2\mathbb{Z}$ be the map given by the composition of $\Phi_{S'}|_{\mathsf{S}(F)},$ where $\Phi_{S'}:F^{\times}\to \prod_{w\in S'}\mathbb{Z}$ is defined in \Cref{funcionhaciaZ}, and the natural surjective map $$\prod_{w\in S'}\mathbb{Z}\to\prod_{w\in S}\mathbb{Z}\to\prod_{w\in S}\mathbb{Z}/2\mathbb{Z}.$$ We claim that $\Phi$ is a surjective group homomorphism with $\ker(\Phi)=\mathsf{S}_{2^{r}}(F).$

First, we observe that $\Phi$ is a group homomorphism simply because valuations and projections are group homomorphisms. The inclusion $\mathsf{S}_{2^{r}}(F)\subseteq \ker(\Phi)$ follows directly from Proposition \ref{doscuadrados}. Let us show that $\ker(\Phi)\subseteq \mathsf{S}_{2^{r}}(F).$ Let $\sigma\in\ker(\Phi)$ and we assume that $\sigma\notin \mathsf{S}_{2^{r}}(F).$ By Proposition \ref{doscuadrados}, $\sigma\notin \mathcal{O}^{\times}F^{\times 2},$ for some $\mathcal{O}\in\mathcal{X}^{r}(F/v).$ Let $w\in S$ be such that $\mathcal{O}_{w}=\mathcal{O},$ and let $a_{1},\ldots,a_{\mathsf{rk}(w)}\in\mathbb{Z}$ be such that $w(\sigma)=(a_{1},\ldots,a_{\mathsf{rk}(w)}).$ Since $w(\sigma)\notin 2\mathbb{Z},$ there exists $d\leq \mathsf{rk}(w)$ such that $a_{d}\notin 2\mathbb{Z}.$ Let $w'=\pi_{d}\circ w.$ Since $w'(\sigma)\notin 2\mathbb{Z},$ the residue field $\kappa_{w'}$ cannot be real by \cite[Lemma 4.1]{BGVG}. Let $\overline{w}$ be the residual valuation of $w$ modulo $w'.$ Then $2^{r}\leq s(\kappa_{w})=s(\kappa_{\overline{w}})\leq s(\kappa_{w'})<\infty,$ which implies that $\mathcal{O}_{w'}\in\mathcal{X}^{r}(F/v).$ Since $S'$ is coherent, we obtain that $w'\in S',$ and since $\mathcal{O}_{w'}\in\mathcal{X}^{r}(F/v),$ we have $w'\in S.$ But $\pi^{1}(w'(\sigma))=a_{d}\notin 2\mathbb{Z},$ which contradicts the fact that $\sigma\in\ker(\Phi).$ This shows that $\ker(\Phi)=\mathsf{S}_{2^{r}}(F).$
 
We show now that $\Phi$ is surjective. Let $(e_{w})_{w\in S}$ be the canonical basis of $\prod_{w\in S}\mathbb{Z}/2\mathbb{Z}$ as a $\mathbb{Z}/2\mathbb{Z}$-module. Consider $w\in S,$ and let $d=\mathsf{rk}(w).$ We claim that there exists $\sigma\in \mathsf{S}(F)$ with $\Phi(\sigma)=e_{w}.$ Since $2^{r}\leq s(\kappa_{w})<\infty,$ there exists $f\in\mathfrak{m}_{w}$ and $x_{1},\ldots,x_{m}\in\mathcal{O}_{w}^{\times},$ for some $m\in\mathbb{N},$ such that $f=1+x_{1}^{2}+\cdots +x_{m}^{2}.$ Let $b=(1-\frac{f}{2})^{2}+x_{1}^{2}+\cdots+ x_{m}^{2}=\frac{f^{2}}{4}.$ Note that $w(b)=2w(f)>e_{1}^{d},$ where $e_{1}^{d}$ is the minimal positive element of $\mathbb{Z}^{d}.$ By \Cref{surjectivity}, there exists $z\in F$ such that for all $\nu\in S$ we have that $\nu(z)=e_{\mathsf{rk}(\nu)+1-d}^{\mathsf{rk}(\nu)}$ if $\nu$ is a refinement of $w$ and $\nu(z)<\min\{0,\nu(f)\}$ otherwise. Let $\sigma=(z-(1-\frac{f}{2}))^{2}+x_{1}^{2}+\cdots +x_{m}^{2}$ in $\mathsf{S}(F).$ Let $\nu\in S$ be a refinement of $w.$ Since $\sigma=z(z-2(1-\frac{f}{2}))+b$ and $\nu(z)<\nu(b)$ we have that $\nu(\sigma)=\nu(z).$ Thus $\pi^{1}(\nu(\sigma))=0$ if $\mathcal{O}_{\nu}$ is a proper refinement of $\mathcal{O}_{w}$ and $\pi^{1}(w(\sigma))=\pi^{1}(e_{1}^{d})=1.$ Consider $\nu\in S$ not a refinement of $w.$ Then $\nu(\sigma)=\nu(z(z-2(1-\frac{f}{2})))=2\nu(z).$ Thus 
$\pi^{1}(\nu(\sigma))\in 2\mathbb{Z}.$ Hence $\Phi(\sigma)=e_{w}.$ Therefore $\Phi$ is surjective and $|G_{r}(F)|=2^{|\mathcal{X}^{r}(F/v)|}.$ 
\end{proof}

A field $K$ is called \emph{hereditarily pythagorean} if $K$ is real and $p(L)=1$ for every finite real extension $L$ of $K.$

\begin{cor}\label{coroneandtwoorderings} Let $r\in\mathbb{N}^{+}$ and $n\in\mathbb{N}.$ Assume that $K$ carries a henselian $\mathbb{Z}^{n}$-valuation $v$ such that $\kappa_{v}$ is hereditarily pythagorean. Let $F/K$ be a function field in one variable. If $|\kappa_{v}^{\times}/\kappa_{v}^{\times 2}|=2^{r},$ then $$|G_{r}(F)|=2^{|\mathcal{X}_{r}(F/v)|}.$$
\end{cor}
\begin{proof} This follows directly from \cite[Corollary 3.3]{BDGMZ} and \Cref{induccionsumas}.  
\end{proof}

\section{Hyperelliptic function fields}\label{optimality} 

A hyperelliptic function field $F/K$ is the function field generated by two variables $X,Y$ subject to the relation $Y^{2}=f(X),$ for some square-free polynnomial $f$ over $K.$ 

In this section, we will apply the valuation-theoric description of the order of $G(F):=G_{1}(F)$ from the previous section in the case where $F/K$ is a hyperelliptic function field where $K$ carries a henselian $\mathbb{Z}^{n}$-valuation $v$ such that $\kappa_{v}$ is hereditarily pythagorean with $\kappa_{v}^{\times}=\kappa_{v}^{\times 2}\cup -\kappa_{v}^{\times 2},$ or equivalently, such that $\kappa_{v}$ is hereditarily euclidean. Under this assumptions, it is a consequence of \cite[Theorem 5.8.2]{BDGMZ} and \cite[Corollary 4.6]{BV09} that $p(F)\leq 3$ (even if $F/K$ is not hyperelliptic). The particular case where $K=\mathbb{R}(\!(t_{1})\!)\ldots(\!(t_{n})\!)$ was showed earlier in \cite[Theorem 6.12]{BGVG}. Thus, in this case we have that $$G(F)=\mathsf{S}_{3}(F)/\mathsf{S}_{2}(F).$$

Let $F/K$ be a hyperelliptic function field. Let $f\in K[X]$ be a square-free polynomial such that $F=K(X)(\sqrt{f}).$ Let $g=\lfloor{\frac{\deg f-1}{2}}\rfloor.$ The integer $g$ is the \textit{genus} of the function field $F.$ See \cite[Proposition 7.4.24]{Liu02}. In \cite[Theorem 3.10]{BV09}, an upper bound for $|G(F)|$ in terms of the order of the square class group of the root fields of the distinct nonreal irreducible factors of $f,$ is given in the general case where $K$ is an arbitrary hereditarily pythagorean field. In the case where $K$ is a field carrying a henselian $\mathbb{Z}^{n}$-valuation $v$ with hereditarily euclidean residue field the bound becomes $$|G(F)|\leq 2^{n(g+1)},$$
see \Cref{remark}. We will then proceed to construct hyperelliptic function fields $F/K$ such that the set $\mathcal{X}^{1}(F/v)$ satisfies $|\mathcal{X}^{1}(F/v)|\geq n(g+1),$ which then immediately yields the equalities $|G(F)|=2^{n(g+1)}$ and $|\mathcal{X}^{1}(F/v)|=n(g+1).$ In particular the upper bound in \cite[Theorem 3.10]{BV09} is optimal. 
\bigskip

\begin{rem}\label{remark} Let $n\in\mathbb{N}$ and let $K$ be a field with a henselian $\mathbb{Z}^{n}$-valuation $v$ and hereditarily euclidean residue field. Then $K$ is a hereditarily pythagorean field, by \cite[Proposition 3.5]{Br76}. Let $f\in K[X]$ be a square-free polynomial and let $F=K(X)(\sqrt{f}).$ Let $K_{1},\ldots,K_{r}$ be the root fields of the distinct nonreal irreducible factors of $f.$ By \cite[Theorem 3.10]{BV09} we have that $$|G(F)|\leq \displaystyle\prod_{i=1}^{r}|K_{i}^{\times}/K_{i}^{\times 2}|.$$
Thus $-1\in K_{i}^{\times 2}$ for all $1\leq i\leq r,$ by \cite[III, Theorem 1]{Be78}. Since $K_{i}/K$ is a finite extension, there exist a henselian $\mathbb{Z}^{n}$-valuation $v'$ on $K_{i}$ such that $\mathcal{O}_{v'}\cap K=\mathcal{O}_{v}$ and $-1\in k_{v'}^{\times 2}.$ It is well-known that $|K_{i}^{\times}/K_{i}^{\times 2}|=2^{n}|\kappa_{v'}^{\times}/\kappa_{v'}^{\times 2}|.$ By \cite[Chap. VII, (7.15)]{Lam} we have that $\kappa_{v'}$ is quadratically closed, whence $|K_{i}^{\times}/K_{i}^{\times 2}|=2^{n},$ for all $1\leq i\leq r.$ Let $g\in\mathbb{N}$ be such that $\deg f=2g+1$ or $2g+2.$ Then $r\leq g+1,$ hence $$|G(F)|\leq 2^{n(g+1)}.$$ 
It follows that the unique situation where the bound can be optimal is when $f\in K[X]$ is a square-free polynomial of degree $2(g+1)$ and has $g+1$ nonreal irreducible factors.  
\end{rem}

\begin{lem}\label{lema2cuadrados} Let $K$ be a real field such that $p(K)=1.$ Let $v$ be a valuation on $K.$ Let $f\in\mathcal{O}_{v}[X]$ be a monic quadratic irreducible polynomial such that $K[X]/(f)$ is nonreal. Then there exist $a,b\in \mathcal{O}_{v}$ such that $f=(X-a)^{2}+b^{2}.$ 
\end{lem}
\begin{proof} Let $\alpha,\beta\in \mathcal{O}_{v}$ be such that $f=X^{2}+\alpha X+\beta.$ Thus, we can write $f=(X-\alpha/2)^{2}+(\beta-\alpha^{2}/4).$ Since $K[X]/(f)=K(\sqrt{\alpha^{2}-4\beta})$ is nonreal, then $-(\alpha^{2}-4\beta)\in \mathsf{S}(K)=K^{\times 2},$ by \cite[Chap. VIII, Lemma 1.4]{Lam}. The statement follows considering $a=\alpha/2$ and $b=\sqrt{(4\beta-\alpha^{2})/4}.$ 

\end{proof}

\begin{lem}\label{irreducibles} Assume that $K$ carries a $\mathbb{Z}$-valuation $v.$ Let $f\in K[X]$ be a non-constant square-free polynomial of degree $d,$ and let $F=K(X)(\sqrt{f}).$ Then $F$ is $K$-isomorphic to $$K(X)(\sqrt{\alpha\cdot q_{1}\cdots q_{r}}),$$ for some  $r\in\mathbb{N},q_{1},\ldots,q_{r}\in \mathcal{O}_{v}[X]$ monic irreducible such that $d=\sum_{i=1}^{r}\deg q_{i},$ and where $\alpha=1$ if $d$ is odd and otherwise $\alpha$ is the leading coefficient of $f.$  
\end{lem}
\begin{proof} Since $\textsf{Frac}(\mathcal{O}_{v})=K,$ we can assume that $f\in\mathcal{O}_{v}[X].$ Let $d=\deg f.$ Let $a_{0},\ldots,a_{d}\in\mathcal{O}_{v}$ be such that $f=\displaystyle\sum_{i=0}^{d}a_{i}X^{i}.$ Since $F$ is the function field of $Y^{2}=f(X),$ multiplying by $a_{d}^{2(d-1)},$ we have that $(a_{d}^{d-1}Y)^{2}=a_{d}^{d-1}g(X),$ where $g=\sum_{i=0}^{d}b_{i}X^{i},$ where $b_{d}=a_{d}^{d}$ and $b_{i}=a_{i}a_{d}^{d-1}$ for all $0\leq i\leq d-1.$ Replacing $X'=a_{d}X,$ and $Y'=a_{d}^{d-1}Y,$ we have that $F$ is $K$-isomorphic to $K(X')\left(\sqrt{a_{d}^{d-1}g(X')}\right).$ Write $g(X')=q_{1}(X')\ldots q_{r}(X'),$ where $q_{i}(X')\in K[X']$ are monic irreducible polynomials, for some $r\in\mathbb{N}.$ Since $\mathcal{O}_{v}$ is a Unique Factorization Domain, by Gauss' Lemma, we may assume that  $q_{1}(X'),\ldots,q_{r}(X')\in\mathcal{O}_{v}[X'],$ which concludes the proof.
\end{proof}

Let $v$ be a valuation on a field $K.$ By \cite[Corollary 2.2.2]{EP}, there exists a unique extension $w$ of $v$ to the rational function field $K(X)$ such that $w(X)=0$ and the residue $\overline{X}\in \kappa_{w}$ is transcendental over $\kappa_{v}.$ A valuation $w$ with these properties is called the \textit{Gauss extension of} $v$ \textit{to} $F$ \textit{with respect to} $X.$

\begin{lem}\label{generoceron} Let $n\in\mathbb{N}$ Assume that $K$ is a field carrying a henselian $\mathbb{Z}^{n}$-valuation with hereditarily euclidean residue field. Let $F/K$ be a regular function field of genus zero. Then
$$|G(F)|=\left\{\begin{array}{ll} 2^{n} & \mbox{ if $F$ is nonreal,}\\ 1 &\mbox{ if $F$ is real.}\end{array}\right.$$ 
\end{lem}

\begin{proof} We assume that $F$ is nonreal. Then $F=K(X)(\sqrt{\alpha\cdot q}),$ where $q$ is a monic irreducible quadratic polynomial over $K,$ by \cite[Theorem 5.7.2, Theorem 5.7.3]{HK00}. Then $q(X)=(X-a)^{2}+b^{2},$ for some $a,b\in K,$ by Lemma \ref{lema2cuadrados}. Moreover $\alpha\in -K^{\times 2},$ because $F$ is nonreal. Replacing $X'=X-a,$ we have that $F$ is isomorphic to $K(X)(\sqrt{-(X^{2}+1)}).$ We prove that statement by induction on $n.$ For $n=0,$ it follows from \cite[Corollary 4.6]{BV09} that $G(F)$ is trivial. Let $n\geq 1.$ Let $v_{1}=\pi_{1}\circ v.$ Let $w'$ be the Gauss extension of $v_{1}$ to $K(X)$ with respect to $X$ and let $w$ be an extension of $w'$ to $F.$ Since $\kappa_{v_{1}}$ is real, the polynomial $X^{2}+1$ is irreducible over $\kappa_{v_{1}}.$ Then we have $\kappa_{w}=\kappa_{v_{1}}(\overline{X})\left(\sqrt{-(\overline{X}^{2}+1)}\right),$ and $\mathcal{O}_{w}\in\mathcal{X}^{1}(F/v)\cap \Omega_{1}(F).$ It follows from \cite[Corollary 3.6]{BGu19} that $\mathcal{O}_{w}$ is the unique valuation ring in $\mathcal{X}^{1}(F/v)\cap \Omega_{1}(F),$ and hence every $\mathcal{O}\in\mathcal{X}^{1}(F/v)$ is a refinement of $\mathcal{O}_{w},$ by Remark \ref{inversevaluations}. Let $\overline{v}$ be the residual valuation of $v$ modulo $v_{1}.$ Thus $|\mathcal{X}^{1}(F/v)|=1+|\{\mathcal{O}\in\Omega(\kappa_{w})\mid \mathcal{O}\in\mathcal{X}^{1}(\kappa_{w}/\overline{v})\}|,$ whereby $|G(F)|=2|G(\kappa_{w})|,$ by \Cref{induccionsumas}. Note that $\overline{v}$ is a henselian $\mathbb{Z}^{n-1}$-valuation on $\kappa_{v_{1}}$ with hereditarily euclidean residue field. Thus, by the induction hypothesis $|G(\kappa_{w})|=2^{n-1}.$ Therefore $|G(F)|=2^{n}.$

If $F$ is real, then $p(F)=2$ by \cite[Theorem 3]{TY03}, whereby $|G(F)|=1.$  
\end{proof}

\begin{ex}\label{exoptimality} Let $g,n\in\mathbb{N}.$ Let $F$ be the function field of the curve $$Y^{2}=-\displaystyle\prod_{i=0}^{g}(X^{2}+t_{n}^{2i})$$ over $K_{n}:=\mathbb{R}(\!(t_{1})\!)\ldots(\!(t_{n})\!).$ We will show that $|G(F)|=2^{n(g+1)}.$ Let $v_{t_{n}}$ be the $t_{n}$-adic valuation on $K_{n}.$ Note that $\mathcal{O}_{v_{t_{n}}}=K_{n-1}[\![t_{n}]\!].$ For $0\leq i\leq g,$ let $v_{i}$ be the Gauss extension of $v_{t_{n}}$ to $K_{n}(X)$ with respect to $Xt_{n}^{-i}$ and let $w_{i}$ be an extension of $v_{i}$ to $F.$ 

We claim that $\kappa_{w_{i}}\simeq K_{n-1}(X)(\sqrt{-(X^{2}+1)}),$ for all $0\leq i\leq g.$ Let $0\leq i\leq g,$ and let $Z:=Xt_{n}^{-i}.$ Hence $F=K_{n}(Z)(\sqrt{-g(Z)})$ since $K_{n}(X)=K_{n}(Z),$ where $g=(t_{n}^{2i}Z^{2}+1)(t_{n}^{2i-2}Z^{2}+1)\cdots(Z^{2}+1)\cdots(Z^{2}+t_{n}^{2(g-i)})$ is a square-free polynomial over $K_{n}.$ Then $g\in \mathcal{O}_{v_{i}}^{\times}$ and $\kappa_{w_{i}}=K_{n-1}(\overline{Z})(\sqrt{-\overline{g}})=K_{n-1}(\overline{Z})\left(\sqrt{-(\overline{Z}^{2}+1)}\right).$

Let $v$ be a henselian $\mathbb{Z}^{n}$-valuation on $K_{n}.$ Let $\overline{v}$ be the residual valuation of $v$ modulo $v_{t_{n}}.$ Since $|\mathcal{X}^{1}(\kappa_{w_{i}}/\overline{v})|=|\{\mathcal{O}\in \Omega(F)\mid \mathcal{O}\subsetneq \mathcal{O}_{w_{i}}\:\textit{and}\:\mathcal{O}\in\mathcal{X}^{1}(F/v)\}|$ by \Cref{inversevaluations} and $|\mathcal{X}^{1}(\kappa_{w_{i}}/\overline{v})|=n-1,$ by \Cref{generoceron}, for every $0\leq i\leq g,$ we have that $|\mathcal{X}^{1}(F/v)|\geq (g+1)+(n-1)(g+1)=n(g+1).$ Therefore $|G(F)|\geq 2^{n(g+1)}$ by \Cref{induccionsumas}, whereby $$|G(F)|=2^{n(g+1)},$$ by \Cref{remark}. 
\end{ex}

Let $F/K$ be a function field in one variable of genus zero, where $K$ is any hereditarily pythagorean field, that is $F=K(X)(\sqrt{aX^{2}+b)}),$ for some $a,b\in K.$ S. Tikhonov and V.I. Yanchevski\u{\i} in \cite[Theorem 3]{TY03} showed that if $F$ is real, then $p(F)=2.$ Hence, if we assume that $K$ carries a henselian $\mathbb{Z}^{n}$-valuation with hereditarily euclidean residue field, for some $n>0,$ then $|G(F)|=2^{n},$ implies that $F$ is a nonreal field of level two, by \Cref{generoceron}. This motivates the following question.

\begin{qu} Let $n,g\in\mathbb{N}.$ Let $K$ be a field carrying a henselian $\mathbb{Z}^{n}$-valuation with hereditarily euclidean residue field. Let $F/K$ be a hyperelliptic function field of genus $g.$ Does the equality $|G(F)|=2^{n(g+1)}$ imply that $F$ is a nonreal field ?
\end{qu} 

We show in \Cref{cotaoptimalg+1enR} that this question has an affirmative answer in the case where $n=1.$

\begin{prop}\label{realhyper} Let $g\in\mathbb{N}.$ Assume that $K$ carries a henselian $\mathbb{Z}$-valuation $v$ with $\kappa_{v}$ hereditarily pythagorean. Let $f\in K[X]$ be a square-free polynomial of degree $2g+2$ with all roots in $K(\sqrt{-1})\setminus K.$ Set $F=K(X)(\sqrt{f}).$ Let $w$ be a residually transcendental extension of $v$ to $F.$ Assume that $F$ is real. Then one of the following conditions holds:	
\begin{itemize}
	\item[(a)] $\kappa_{w}$ is nonreal with $s(\kappa_{w})=1.$ 
	\item[(b)] $\kappa_{w}/\kappa_{v}$ is ruled.
\item[(c)] $\kappa_{w}=\kappa_{v}(X)(\sqrt{h})$ is a real field with $h\in \kappa_{v}[X]$ a square-free polynomial with all roots in $\kappa_{v}(\sqrt{-1})\setminus \kappa_{v}.$
\end{itemize}
\end{prop}

\begin{proof} By Lemma \ref{lema2cuadrados} and by Lemma \ref{irreducibles} we may choose irreducible polynomials $q_{i}=(X-a_{i})^{2}+b_{i}^{2}$ with $a_{i},b_{i}\in\mathcal{O}_{v},$ and $\alpha\in K^{\times}$ such that $f=\alpha\cdot q_{0}\cdots q_{g}.$ Since $F$ is real, $\alpha\notin -K^{\times 2}.$ We set $Y_{i}=(X-a_{i})b_{i}^{-1}.$ 

Let $w$ be a residually transcendental extension of $v$ to $F$ and let $w_{0}=w|_{K(X)}.$ As $[F:K(X)]=2$ we have that $[\kappa_{w}:\kappa_{w_{0}}]\leq 2.$ Let $\ell$ be the relative algebraic closure of $\kappa_{v}$ in $\kappa_{w}.$ If $\ell$ is nonreal, since $\kappa_{v}$ is a hereditarily pythagorean, we have that $-1\in \kappa_{w}^{\times 2}$ (case $(a)$). 

Thus we assume now that $\ell$ is real. If $\kappa_{w}=\kappa_{w_{0}},$ since $w$ is a residually transcendental extension of $v,$ by Proposition \ref{Ohm} we have that $\kappa_{w}=\kappa_{w_{0}}=\ell (\overline{T}),$ for some $T\in\mathcal{O}_{w_{0}}^{\times}$ with $\overline{T}$ transcendental over $\kappa_{v}$ (case $(b)$). Now assume that $[\kappa_{w}:\kappa_{w_{0}}]=2$ and $\ell$ is real. Then $\Gamma_{w}=\Gamma_{w_{0}}.$ Let $f'=q_{0}\cdots q_{g}.$ Since $f\in F^{\times 2},$ we have that $w_{0}(f)\in 2\Gamma_{w_{0}}$ and since $f'\in \mathsf{S}_{2}(K(X))$ we have $w_{0}(f')\in 2\Gamma_{w_{0}},$ by \cite[Lemma 4.1]{BGVG}. Hence $w_{0}(\alpha)\in 2\Gamma_{w_{0}}.$ Consider $i\in\{0,\ldots,g\}.$ Set
$$Z_{i}=\left\{\begin{array}{ll} 1+Y_{i}^{2} & \mbox{if $w(Y_{i})\geq 0$,}\\ 1+Y_{i}^{-2} &\mbox{ if $w(Y_{i})<0$.}\end{array}\right.$$
Note that, since $\ell$ is real, $Z_{i}\in\mathcal{O}_{w}^{\times}$ for all $0\leq i\leq g.$ Thus 
$$\overline{Z_{i}}=\left\{\begin{array}{ll} 1+\overline{Y_{i}}^{2} & \mbox{if $w(Y_{i})=0$,}\\ \overline{1} &\mbox{ otherwise .}\end{array}\right.$$ 

Let us first assume that $v(\alpha)\notin 2\mathbb{Z}.$ This implies that $w_{0}$ is a ramified extension of $v.$ Hence $\overline{Y}_{i}$ is algebraic over $\kappa_{v}$ for all $i\in\{0,\ldots,g\},$ because otherwise $w_{0}$ must be a Gauss extension of $v$ with respect to some $Y_{i},$ which would contradict the ramification. Since $\ell$ is pythagorean, we have $\overline{Z_{i}}\in\ell^{\times 2}\subseteq \kappa_{w}^{\times 2}$ because $\ell$ is the relative algebraic closure of $\kappa_{v}$ in $\kappa_{w}.$ By Proposition \cite[Lemma 2.2]{BGu19} we have that $\kappa_{w}=\kappa_{w_{0}}(\sqrt{\overline{u}}),$ for any $u\in \alpha f'K(X)^{\times 2}$ with $w_{0}(u)=0.$ Since $\prod_{i=0}^{g}\overline{Z_{i}}\in \kappa_{w_{0}}^{\times 2},$ we have that $\kappa_{w}=\kappa_{w_{0}}\left(\sqrt{\overline{\alpha h^{2}}}\right)$ for any $h\in K(X)^{\times}$ such that $w_{0}(\alpha h^{2})=0.$ Let $K'=K(\sqrt{\alpha}).$ Let $w_{0}'$ be an extension of $w_{0}$ to $K'(X)$ and let $v'=w_{0}'|_{K'}.$ We have that $\kappa_{w_{0}'}=\kappa_{w_{0}}\left(\sqrt{\overline{\alpha h^{2}}}\right),$ because $K'(X)=K(X)(\sqrt{\alpha})$ and $w_{0}(\alpha h^{2}).$ Furthermore, $w_{0}'$ is a residually transcendental extension of $v'$ to $K'(X).$ Therefore $\kappa_{w}/\kappa_{v}$ is ruled, by Proposition \ref{Ohm} (case $(b)$).


Let us assume now that $v(\alpha)\in 2\mathbb{Z}.$ Let $J$ be the set of indices $i\in\{0,\ldots,g\}$ such that $w(Y_{i})=0$ and $\overline{Y}_{i}$ is transcendental over $\kappa_{v}.$ Assume first that $J=\emptyset.$ Then $\overline{Z_{i}}\in\kappa_{w}$ is algebraic over $\ell$ for all $0\leq i\leq g,$ and since $\ell$ is pythagorean, we have that $\kappa_{w}=\kappa_{w_{0}}(\sqrt{\alpha}).$ Therefore $\kappa_{w}/\kappa_{v}$ is ruled (case ($b$)). Now we assume that $J\neq \emptyset.$ Without loss of generality we put $J=\{0,\ldots,s\},$ for some $s\leq g.$ For $i\in J,$ let $v_{i}$ be the Gauss extension of $v$ to $K(X)$ with respect to $Y_{i}.$ By \cite[Corollary 2.2.2]{EP} we have that $w_{0}=v_{0}=\cdots =v_{s},\;\kappa_{w_{0}}=\kappa_{v}(\overline{Y_{0}})$ and $\Gamma_{w_{0}}=\mathbb{Z}.$ Thus $w_{0}(\alpha)\in 2\mathbb{Z},$ and we may consider some $\beta\in\mathcal{O}_{v}^{\times}$ such that $\alpha\in \beta K^{\times 2}.$ Let $j\in J.$ We have $Y_{j}^{2}+1\in\mathcal{O}_{w_{0}}^{\times}.$ We set $c_{j}=b_{0}b_{j}^{-1},d_{j}=(a_{0}-a_{j})b_{j}^{-1}.$ Since $Y_{j}=c_{i}Y_{0}+d_{j},$ we have $Y_{j}^{2}+1=c_{j}^{2}Y_{0}^{2}+2c_{j}d_{j}Y_{0}+d_{j}^{2}+1,$ and hence $f\in h(Y_{0})\cdot K^{\times 2},$ where $$h(Y_{0})=\beta(Y_{0}^{2}+1)\cdots(c_{s}^{2}Y_{0}^{2}+2d_{s}c_{s}Y_{0}+d_{s}^{2}+1)\displaystyle\prod_{i=s+1}^{g}Z_{i}.$$ Therefore $F=K(Y_{0})(\sqrt{h(Y_{0})}).$ Since $v$ is henselian and $\beta\notin -K^{\times 2},$ we have $\overline{\beta}\notin -\kappa_{v}^{\times 2}.$ Finally, since $h\in\mathcal{O}_{w_{0}}^{\times},$ we have that $h\in fK(Y_{0})^{\times 2}\cap\mathcal{O}_{w_{0}}^{\times},$ and we can conclude that $\kappa_{w}=\kappa_{v}(\overline{Y_{0}})\left(\sqrt{\overline{h}(\overline{Y_{0}})}\right)$ is a real field, where $\overline{h}\in \kappa_{v}[\overline{Y_{0}}]$ is a polynomial with all roots in $\kappa_{v}(\sqrt{-1})\setminus \kappa_{v}$ by Proposition \cite[Lemma 2.2]{BGu19} (case $(c)$).
\end{proof}

\begin{cor}\label{sumadedoscuadradosporunt} Let $n\in\mathbb{N}^{+}.$ Assume that $K$ carries a henselian $\mathbb{Z}^{n}$-valuation $v$ such that $\kappa_{v}$ is hereditarily euclidean. Let $f\in K[X]$ be a nonconstant square-free polynomial with all roots in $K(\sqrt{-1})\setminus K.$ We set $F=K(X)(\sqrt{f}).$ Assume that $F$ is real. Then $p(F)=2.$ 
\end{cor}
\begin{proof} We prove the statement by induction on $n.$ Note that by Corollary \ref{coroneandtwoorderings} we have that $p(F)=2$ if and only if $\mathcal{X}^{1}(F/v)$ is empty. Assume $n=1.$ Let $w$ be a residually transcendental extension of $v.$ It follows by Proposition \ref{realhyper} that either $s(\kappa_{w})=1$ or $\kappa_{w}$ is real. Hence $\mathcal{X}^{1}(F/v)=\emptyset$ and $p(F)=2.$ Assume now that $n>1.$ We show that $\mathcal{X}^{1}(F/v)$ is empty. Let $v_{1}=\pi_{1}\circ v.$ Let $\overline{v}$ be the residual valuation of $v$ modulo $v_{1}.$ Then $\overline{v}$ is a henselian $\mathbb{Z}^{n-1}$-valuation on $\kappa_{v_{1}}$ such that $\kappa_{v}=\kappa_{\overline{v}}.$ It follows by induction hypothesis and by \Cref{realhyper} that all the residually transcendental extensions $w$ of $v_{1}$ to $F$ satisfy $p'(\kappa_{w})=2.$ If we had some $\mathcal{O}\in\mathcal{X}^{1}(F/v),$ then we would obtain that the residue field of the rank-one coarsening $\mathcal{O}'$ of $\mathcal{O}$ would have $p'(\kappa_{\mathcal{O}'})>2$ by Proposition \ref{nivelpythagoras}, which is a contradiction. Therefore $\mathcal{X}^{1}(F/v)=\emptyset,$ whereby $p(F)=2$ by \Cref{induccionsumas}. 
\end{proof}

Note that, if $f\in K[X]$ is assumed to be monic in \Cref{sumadedoscuadradosporunt}, then $F/K(X)$ is a totally positive quadratic extension and hence the above result follows by \cite[Corollary 4.10]{BV09}.

\begin{cor}\label{cotaoptimalg+1enR} Let $g\in\mathbb{N}.$ Assume that $K$ carries a henselian $\mathbb{Z}$-valuation $v$ with hereditarily euclidean residue field. Let $F/K$ be a hyperelliptic function field of genus $g.$ Let $f\in K[X]$ be a square-free polynomial such that $F=K(X)(\sqrt{f}).$ If $|G(F)|=2^{g+1},$ then $F$ is nonreal. 
\end{cor}

\begin{proof} It follows by Proposition \ref{remark} and by Lemma \ref{lema2cuadrados} that we may choose irreducible polynomials $q_{i}=(X-a_{i})^{2}+b_{i}^{2},$ with $a_{i},b_{i}\in\mathcal{O}_{v},$ and $\alpha\in K^{\times}$ such that $f=\alpha\cdot q_{0}\cdots q_{g}.$ For $i\in\{0,\ldots,g\}$ let $Y_{i}:=b_{i}^{-1}(X-a_{i}).$ If $\alpha\notin -K^{\times 2},$ then $p(F)=2,$ by Corollary \ref{sumadedoscuadradosporunt}. Therefore, under the assumption that $|G(F)|=2^{g+1},$ we have that $\alpha\in -K^{\times 2}.$ Since $\mathsf{S}_{2}(F)$ is a group, there exists $\eta_{1},\eta_{2}\in K(X)$ such that $F=K(X)(\sqrt{-(\eta_{1}^{2}+\eta_{2}^{2})}),$ that is, $F$ is nonreal.
\end{proof}

\begin{qu} Let $n,g\in\mathbb{N}.$ Assume that $K$ carries a henselian $\mathbb{Z}^{n}$-valuation with hereditarily euclidean residue field. Let $F/K$ be a function field in one variable of genus $g\in\mathbb{N}.$ Is $|G(F)|\leq 2^{n(g+1)}$ for every function field in one variable $F/K$ of genus $g$?

\end{qu} 

\subsection*{Acknowledgments} I thank Marco Zaninelli, Parul Gupta and Nicolas Daans for useful conversation on these topics. I thank my thesis advisors David Grimm and Karim Johannes Becher. 

This work was supported by CONICYT-PFCHA/Doctorado Nacional/2017-folio 21170477, and by Universidad de Santiago de Chile (Proyecto DICYT, c\'odigo 041933G), by the FWO Odysseus Programme (project G0E6114N), and by the Bijzonder Onderzoeksfonds (BOF), Universiteit Antwerpen, (project BOF-DOCPRO4, 2533).


\bibliographystyle{plain}

\begin{thebibliography}{9}



\bibitem{BDGMZ}
K.J.~Becher, N.~ Daans, D.~Grimm, G. Manzano-Flores, M. Zaninelli.\emph{The Pythagoras number of a rational function field in two variables.}
\url{https://arxiv.org/abs/2302.11425} (2023).





\bibitem{BG20}
K.J. Becher, D.~Grimm, \emph{Nonsplit conics in the reduction of an arithmetic curve}, 
https://arxiv.org/abs/2005.11855.

\bibitem{BGVG}
K.J. Becher, D.~Grimm, and J.~Van~Geel, \emph{Sums of squares in algebraic
  function fields over a complete discretely valued field}, Pacific J.~of Math
  \textbf{267} (2014), 257--276.

\bibitem{BGu19}
K.J. Becher, P.~Gupta, \emph{Ruled residue theorem for function fields of conics.} Journal of Pure and Applied Algebra, \textbf{225(06)} (2021) 106638. 

\bibitem{BV09} K.J. Becher and J. Van Geel, \emph{Sums of squares in function fields of hyperelliptic curves}, Mathemastische Zeitschrift, 261 (4): (2009) 829 - 844.


\bibitem{Be78}
Becker, E. \emph{Hereditarily pythagorean fields and orderings of higher level.} Monografias de Matem\'atica Vol. 29. Instituto de matematica pura e aplicada, Rio de Janeiro (1978).

\bibitem{Br76}
L. Br\"ocker, \emph{Characterization of fans and hereditarily pythagorean fields.} Math. Z. \textbf{152}, (1976), 149-163.


\bibitem{EP} A.J. Engler, A.~Prestel, \emph{Valued fields}, Springer-Verlag, 2005.




\bibitem{GR16}
D.~Grimm, \emph{On an isotropy criterion for quadratic forms over function fields of curves over non-dyadic complete discrete valuation rings}. Algebra, logic and number theory, 95–103, Banach Center Publ., 108, Polish Acad. Sci. Inst. Math., Warsaw, (2016).

\bibitem{HHK09}
D. Harbater, J. Hartmann, and D. Krashen, \emph{Applications of patching to quadratic forms and central simple algebras}, Inventiones Mathematicae \textbf{178} (2009), 231-269.

\bibitem{Jac64}
N.~Jacobson, \emph{Lectures in Abstract Algebra, III. Theory of Fields and Galois Theory.} Springer-Verlag, Berlin.

\bibitem{HK00}
H. Koch. \emph{Number theory: Algebraic numbers and functions.} Graduate Studies in Mathematics, 2000.

\bibitem{Lam}
T.Y. Lam, \emph{Introduction to quadratic forms over fields}, American
  Mathematical Society, 2005.

\bibitem{Liu02}
Q. Liu, \emph{Algebraic Geometry and Arithmetic curves}, Oxford Graduate Texts in Mathematics, Oxford University Press, Oxford, 2002.



\bibitem{ohm}
J.~Ohm, \emph{The ruled residue theorem for simple transcendental extensions of
  valued fields}, Proceedings of the American Mathematical Society \textbf{89}
  (1983), no.~1, 16--18.
 
\bibitem{Rib57}
P.~Ribenboim, \emph{Le th\'eor\`eme d' approximation pour les valuations de krull}. Math. Z. \textbf{68}, 1-18 (1957).

\bibitem{TVY06}
S.V.~Tikhonov, J.~Van Geel, V.I~Yanchevski\u{\i}, \emph{Pythagoras number of function fields of hyperelliptic curves with good reduction}, Manuscripta Math. \textbf{119} (2006), 305-322. 

\bibitem{TY03}
S.V~Tikhonov, V.I~Yanchevski\u{\i}, \emph{Pythagoras number of function fields of conics over hereditarily pythagorean fields.} Dokl. Nats. Akad. Nauk Belarusi \textbf{47}, 5-8 (2003).

\bibitem{Wa89}
S. Warner, \emph{Topological fields}. Mathematics studies \textbf{157}, North Holland, Amsterdam (1989).

\bibitem{Witt34}
E.~Witt, \emph{Zerlegung reeller algebraischer Funktionen in Quadrate. Schiefk\"orper \"uber reellem Funktionenk\"orper}. J. Reine Angew. math \textbf{171} (1934), 4-11.

\end{thebibliography}

\end{document}